\documentclass[a4paper,10pt]{amsart}
\usepackage{amsmath,amsthm,amssymb,enumerate}
\usepackage{epsfig}
\usepackage{amssymb}
\usepackage{amsmath}
\usepackage{amssymb}
\usepackage{amsmath,amsthm}
\usepackage[latin1]{inputenc}
\usepackage{bbm}
\usepackage{amsfonts}
\usepackage{amsxtra}
\usepackage{euscript,mathrsfs}
\usepackage{color}
\usepackage[left=4cm,right=4cm,top=4cm,bottom=3cm]{geometry}
\usepackage{xcolor}
\colorlet{ColorPink}{red!30}
\definecolor{Gump}{rgb}{0,0.6,0.4}
\usepackage{graphicx}

\usepackage[colorlinks=true, linktocpage=true, linkcolor=red!70!black,
citecolor=green!50!black]{hyperref}
\allowdisplaybreaks
\usepackage{tikz-cd}
\usepackage{pgfplots}

\usepackage{pst-plot}
\psset{unit=5mm,plotpoints=1000,algebraic}

\usepackage{enumitem}
\setenumerate{label={\rm (\alph{*})}}

\usepackage{amsfonts}
\usepackage{amsxtra}

\theoremstyle{theorem}

\newtheorem{theorem}{Theorem}[section]
\newtheorem{lemma}[theorem]{Lemma}
\newtheorem{proposition}[theorem]{Proposition}
\newtheorem{corollary}[theorem]{Corollary}

\newtheorem{definition}[theorem]{Definition}

\theoremstyle{remark}

\newtheorem{remark}[theorem]{Remark}
\newtheorem{example}[theorem]{Example}

\newcommand{\R}{\mathbb R}
\newcommand{\N}{\mathbb N}

\newcommand{\dx}{\,\mathrm{d}x}
\newcommand{\dt}{\,\mathrm{d}t}

\newcommand{\dHaus}{\,\mathrm{d}\mathscr{H}}

\newcommand{\dif}{\mathrm{d}}

\newcommand{\curl}{\mathrm{curl}}

\renewcommand{\dif}{\operatorname{d}\!}
\newcommand{\lebe}{\operatorname{L}}
\newcommand{\sobo}{\operatorname{W}}
\newcommand{\hilb}{\operatorname{H}}

\newcommand{\imag}{\operatorname{i}}
\newcommand{\locc}{\operatorname{loc}}

\newcommand{\hold}{\operatorname{C}}

\newcommand{\ball}{\operatorname{B}}
\newcommand{\di}{\operatorname{div}}
\newcommand{\A}{\mathbb{A}}
\newcommand{\B}{\mathbb B}

\newcommand{\id}{\operatorname{Id}}

\newcommand{\T}{\operatorname{T}}
\renewcommand{\phi}{\varphi}
\newcommand{\C}{\mathbb{C}}
\newcommand{\Image}{\mathrm{Im}}
\newcommand{\spano}{\mathrm{span}}
\newcommand{\Lin}{\mathscr{L}}
\renewcommand{\H}{\mathrm{H}}
\newcommand{\lbq}{\operatorname{L}_{\mathbb{B}}^2(Q)}
\newcommand{\rsym}{\mathbb{R}_{\operatorname{sym}}^{n\times n}}

\allowdisplaybreaks

\newcommand{\torus}{\mathbb{T}}

\numberwithin{equation}{section}

\begin{document}
	
	
	\title[Operators with constant rank over $\mathbb{C}$]{Natural annihilators and \\ operators of constant rank over $\mathbb{C}$}

	\author[F.~Gmeineder]{Franz Gmeineder}
	\address[F.~Gmeineder]{Fachbereich Mathematik, Universit\"{a}t Konstanz, Universit\"{a}tsstra\ss e 10, Konstanz, Germany}
	\author[S.~Schiffer]{Stefan Schiffer}
	\address[S.~Schiffer]{Institut f\"{u}r angewandte Mathematik, Universit\"{a}t Bonn, Endenicher Allee 60, 53115 Bonn, Germany}

	\maketitle
	\begin{abstract}
	Even if the Fourier symbols of two constant rank differential operators have the same nullspace for each non-trivial phase space variable, the nullspaces of those differential operators might differ by an infinite dimensional space. Under the natural condition of constant rank over $\mathbb{C}$, we establish that the equality of nullspaces on the Fourier symbol level already implies the equality of the nullspaces of the differential operators in $\mathscr{D}'$ modulo polynomials of a fixed degree. In particular, this condition allows to speak of \emph{natural annihilators} within the framework of complexes of differential operators. As an application, we establish a Poincar\'{e}-type lemma for differential operators of constant complex rank in two dimensions. 
	\end{abstract}
	\section{Introduction} 
	\subsection{Aim and scope}
	Let $V,W,X$ be three real, finite dimensional inner product spaces and let, for $k,\ell\in\mathbb{N}$, 
	\begin{align}\label{eq:diffops}
	\A := \sum_{|\alpha|=k}\A_{\alpha}\partial^{\alpha},\;\;\;\mathbb{B}:=\sum_{|\beta|=\ell}\mathbb{B}_{\beta}\partial^{\beta}
	\end{align}
	be two constant coefficient differential operators on $\R^{n}$ from $V$ to $W$ or from $W$ to $X$, respectively. By this we understand that for each $|\alpha|=k$ and $|\beta|=\ell$, we have $\A_{\alpha}\in\mathscr{L}(V;W)$ or $\mathbb{B}_{\beta}\in\mathscr{L}(W;X)$. 
	
	For instance, this setting comprises the usual gradient $Du$ for maps $u\colon\R^{n}\to\R^{N}$ or the symmetric gradient $\varepsilon(u):=\frac{1}{2}(Du+Du^{\top})$ for maps $u\colon\R^{n}\to\R^{n}$ as frequently employed in nonlinear elasticity; these can be recovered by the particular choices $(V,W)=(\R^{N},\R^{N\times n})$ or $(V,W)=(\R^{n},\rsym)$, respectively. To describe the main question of the present paper, note that 
	\[ \begin{tikzcd}
	\hold^{\infty}(\R^{n};\R^{N}) \arrow{r}{D} & \hold^{\infty}(\R^{n};\R^{N\times n}) \arrow{r}{\mathrm{curl}} & \hold^{\infty}(\R^{n};\R^{N\times n}),
	\end{tikzcd}\] \[
	\begin{tikzcd}
	\hold^{\infty}(\R^{n};\R^{n}) \arrow{r}{\varepsilon} & \hold^{\infty}(\R^{n};\rsym) \arrow{r}{\mathrm{curlcurl}} & \hold^{\infty}(\R^{n};\R^{d}),
	\end{tikzcd}\]
	for suitable $d\in\mathbb{N}$, are sequences that are exact  at the corresponding mid point vector spaces. Here, we have set for $u=(u_{jk})_{1\leq j \leq N,\,1\leq k\leq n}$ and $v=(v_{jk})_{1\leq j,k\leq n}$ 
	\begin{align}\begin{split}
	&\curl(u)=\Big(\partial_{k}u_{ji}-\partial_{i}u_{jk} \Big)_{ijk},\\ & \curl\curl^{\top}(v)=\Big(\partial_{ij}v_{kl}+\partial_{kl}v_{ij}-\partial_{il}v_{kj}-\partial_{kj}v_{il} \Big)_{ijkl}.
	\end{split}
	\end{align}
	The first example is the usual gradient-curl-complex, whereas the second one is referred to as the Saint-Venant compatibility complex (see, e.g., \cite{Ciarlet}). In the language of Fourier analysis, this circumstance can be restated by the associated symbol complex 
	\[ \begin{tikzcd}
	V \arrow{r}{\mathbb{A}[\xi]} & W \arrow{r}{\mathbb{B}[\xi]} & X
	\end{tikzcd} \qquad \text{being exact at $W$ for all $\xi\in\R^{n}\setminus\{0\}$}\]
	for the corresponding choices $(\mathbb{A},\mathbb{B})=(D,\mathrm{curl})$ or $(\mathbb{A},\mathbb{B})=(\varepsilon,\mathrm{curlcurl}^{\top})$, respectively. This means that for each $\xi\in\R^{n}\setminus\{0\}$ we have $\mathbb{A}[\xi](V)=\ker(\mathbb{B}[\xi])$, where 
	\begin{align}\label{eq:form1}
	\A[\xi]=\sum_{|\alpha|=k}\xi^{\alpha}\mathbb{A}_{\alpha},\;\;\;\mathbb{B}[\xi]=\sum_{|\beta|=\ell}\xi^{\beta}\mathbb{B}_{\beta},\qquad \xi\in\R^{n}. 
	\end{align}
	In this situation we call $\mathbb{A}$ a \emph{potential} of $\mathbb{B}$, and $\mathbb{B}$ an \emph{annihilator} of $\mathbb{A}$. Annihilators are far from being uniquely determined: For instance, letting $\Delta$ be the usual Laplacian, each $\Delta^{j}\mathbb{B}$ ($j\in\mathbb{N}$) satisfies $\ker(\Delta^{j}\mathbb{B}[\xi])=\ker(|\xi|^{2j}\mathbb{B}[\xi])=\ker(\mathbb{B}[\xi])$ for any $\xi\in\R^{n}\setminus\{0\}$. Still, in the above examples with $(\mathbb{A},\mathbb{B})=(D,\curl)$ or $(\mathbb{A},\mathbb{B})=(\varepsilon,\curl\curl^{\top})$, the nullspaces of $\mathbb{B}$ and $\Delta^{j}\mathbb{B}$ differ by an infinite dimensional vector space. Thus, denoting the class of annihilators of a given differential operator $\A$ by 
	\begin{align}\label{eq:annihilatorclass}
	\mathrm{An}(\mathbb{A}):=\left\{\mathbb{B}\colon\;\begin{array}{c} \text{$\mathbb{B}$ is of the form~\eqref{eq:form1} for some vector space $X$,} \\ \ker(\mathbb{B}[\xi])=\mathbb{A}[\xi](V)\;\;\text{for all}\;\xi\in\R^{n}\setminus\{0\} \end{array}  \right\}, 
	\end{align}
	it is logical to ask for a subset $\mathcal{C}\subset \mathrm{An}(\mathbb{A})$ with the property that the distributional nullspaces of $\mathbb{B},\mathbb{B}'\in\mathcal{C}$ only differ by a finite dimensional vector space each, and under which conditions on $\mathbb{A}$ the class $\mathcal{C}$ is non-empty. 
	\subsection{Operators with constant rank over $\mathbb{C}$} 
	We first recall some terminology that is customary in the above context. Following the works of \textsc{Wilcox \& Schulenberger} \cite{WiSc} and \textsc{Murat} \cite{Murat} (also see \textsc{Fonseca \& M\"{u}ller} \cite{FoMu}), operators $\mathbb{A}$ or $\mathbb{B}$ of the form~\eqref{eq:form1} are said to be of \emph{constant rank (over $\R$)} provided $\dim_{\R}(\mathbb{A}[\xi](V))$ or $\dim_{\R}(\mathbb{B}[\xi](W))$, respectively, are independent of the phase space variable $\xi\in\R^{n}\setminus\{0\}$. By \textsc{Raita} \cite{Raita} (also see \cite{ArRaSi}), every constant rank operator possesses a constant rank potential. 
	
	Towards the above question from ~\eqref{eq:annihilatorclass}\emph{ff.}, a strengthening of the notion of constant rank is required:
	\begin{definition}[Constant rank over $\mathbb{C}$]\label{def:CCR}
		Let $\mathbb{B}$ be a differential operator as in \eqref{eq:diffops}. We say that $\mathbb{B}$ has \emph{constant rank over $\mathbb{C}$} provided 
		\begin{align}\label{eq:CCR}
		\dim_{\mathbb{C}}(\mathbb{B}[\xi](V+\mathrm{i}V))\quad\text{is independent of}\;\xi\in\mathbb{C}^{n}\setminus\{0\}. 
		\end{align}
	\end{definition}
	For~\eqref{eq:CCR}, note that whenever a complex phase variable $\xi=\mathrm{Re}(\xi)+\imag\mathrm{Im}(\xi)$ is inserted into $\mathbb{B}[\xi]$, it consequently gives rise to a linear map $\mathbb{B}[\xi]\colon V+\mathrm{i}V\to W+\mathrm{i}W$. Similarly as the constant rank operators generalise the notion of (overdetermined real) elliptic differential operators $\mathbb{A}$ \'{a} la \textsc{H\"{o}rmander} and \textsc{Spencer} \cite{Hoermander,Spencer}, operators of constant rank over $\mathbb{C}$ generalise the concept of $\mathbb{C}$-elliptic operators in the spirit of \textsc{Smith} \cite{Smith} (also see \cite{BDG,GR19,Kala1}). Here, an operator $\mathbb{A}$ is called \emph{(real or $\R$-) elliptic} provided $\mathbb{A}[\xi]\colon V\to W$ is injective for all $\xi\in\R^{n}\setminus\{0\}$, and \emph{$\mathbb{C}$-elliptic} provided $\mathbb{A}[\xi]\colon V+\imag V\to W+\imag W$ is injective for all $\xi\in\mathbb{C}^{n}\setminus\{0\}$. Adopting the terminology of Definition~\ref{def:CCR}, the main result of the present paper is as follows: 
	\begin{theorem}\label{thm:main}
		Let $\B$, $\widetilde{\B}$ be two differential operators with constant rank over $\C$. Then the following are equivalent:
		\begin{enumerate} 
			\item\label{item:main1} For all $\xi \in \C^n \setminus \{0\}$ we have
			\begin{align*}
			\ker(\B[\xi]) = \ker(\widetilde{\B}[\xi]).
			\end{align*}
			\item\label{item:main2} There exist two finite dimensional vector subspaces $\mathcal{X}_{1},\mathcal{X}_{2}$ of the $W$-valued polynomials on $\R^{n}$ such that 
			\begin{align}\label{eq:finitedim}
			\ker(\mathbb{B})+\mathcal{X}_{1}=\ker(\widetilde{\mathbb{B}})+\mathcal{X}_{2}, 
			\end{align}
			where $\ker$ is understood as the nullspace in $\mathscr{D}'(\R^{n};W)$, so e.g.
			\begin{align*}
			\ker(\mathbb{B})=\{T\in\mathscr{D}'(\R^{n};W)\colon\;\mathbb{B}T=0\}.
			\end{align*}
		\end{enumerate}
	\end{theorem}
	Let us note that if the Fourier symbols $\mathbb{B}[\xi]$ and $\widetilde{\mathbb{B}}[\xi]$ have the same nullspace for any $\xi$, then they are both annihilators of some differential operator $\A$ with constant rank in $\C$. Also note that the statement of Theorem \ref{thm:main} is false if we drop the assumption that $\B$ and $\widetilde{\B}$ satisfy the constant rank property over $\C$ (cf. Example \ref{Bilaplace}).
	
	In the language of algebraic geometry, the proof of Theorem~\ref{thm:main} relies on a vectorial Nullstellensatz to be stated and established in Section~\ref{sec:nullstellensatz} below. Nullstellensatz techniques have been employed in slightly different contexts (see \cite{Smith,Kala1,GRV21}). However, these by now routine applications to differential operators (to be revisited in detail in Section~\ref{sec:nullstellensatz}) do not prove sufficient to establish Theorem~\ref{thm:main}. 
	
	If a differential operator $\mathbb{A}$ has an annihilator $\mathbb{B}$ of constant complex rank, this annihilator is in some sense minimal when being compared with other annihilators (so e.g. $\mathbb{D}\circ\mathbb{B}$ for (real) elliptic operators $\mathbb{D}$ on $\R^{n}$ from $X$ to some finite dimensional real vector space $Y$). Thus, annihilators of constant complex rank -- provided existent -- are \emph{natural}. Even though the condition of constant rank over $\mathbb{C}$ appears quite restrictive, it is satisfied for a wealth of operators to be gathered below. As an interesting byproduct, such annihilators can be utilised to derive a Poincar\'{e}-type lemma in $n=2$ dimensions; see Section~\ref{sec:application} for this matter and related open questions in this context.
	
	We wish to point out that when preparing this note, a variant of the above theorem was also established in \textsc{H\"{a}rk\"{o}nen, Niklasson \& Raita} \cite[Thm.~1.2]{HNR}. However, the techniques employed to arrive at this statement are different from ours; specifically, our proof only hinges on the Hilbert Nullstellensatz and elementary linear algebra.
 
	\subsection{Organisation of the document}
	Apart from this introductory section, the paper is organised as follows: In Section~\ref{sec:examples} we gather examples of operators arising in applications that verify the constant rank condition over $\mathbb{C}$. Section~\ref{sec:nullstellensatz} then is devoted to a suitable variant of a vectorial Nullstellensatz, that displays the pivotal step in the proof of Theorem~\ref{thm:main} in Section~\ref{sec:main}. The paper then is concluded by a sample application on a two-dimensional Poincar\'{e}-type lemma in Section~\ref{sec:application}.
	\subsection{Notation}
	For $k\in\mathbb{N}$, we denote $\mathscr{P}_{k}(\R^{n};\R^{d})$ the $\R^{d}$-valued polynomials on $\R^{n}$ of degree at most $k$; the space of $\R^{d}$-valued polynomials $p$ on $\R^{n}$ which are homogeneous of degree $k$, so satisfy $p(\lambda x)=\lambda^{k}p(x)$ for all $\lambda\in\R$ and $x\in\R^{n}$, is denoted as $\mathscr{P}_{k}^{h}(\R^{n};\R^{d})$. Moreover, given a ring $\mathcal{R}$, we use the convention $I\unlhd\mathcal{R}$ to express that $I$ is an ideal in $\mathcal{R}$.
	{\small
		\subsection*{Acknowledgments.} The authors are thankful to Bogdan Raita for useful comments on a preliminary version of the paper. The authors are moreover grateful for fincancial support through the  Hector foundation (Project-Nr. 11962621) (F.G.) and to the DFG through the graduate school BIGS of the 
Hausdorff Center for Mathematics (GZ EXC 59 and 2047/1, Projekt-ID 390685813) (S.Sc.).}
	
	\section{(Non-)Examples of operators of constant rank over $\C$}\label{sec:examples}
	In this section we discuss some (non-)examples  that satisfy the algebraic condition of constant rank over $\C$ from Definition~\ref{def:CCR} and arise frequently in applications. 
	\begin{example}[$\mathbb{C}$-elliptic operators]
		$\mathbb{C}$-ellipticity of an operator $\mathbb{A}$ of the form \eqref{eq:form1} means that $\mathbb{A}[\xi]\colon V+\imag\!V\to W+\imag\!W$ is injective for any $\xi\in\mathbb{C}^{n}\setminus\{0\}$. Such operators have constant rank over $\C$ by definition; trivially, the usual $k$-th order gradients are $\C$-elliptic. As discussed e.g. in \cite[Ex.~2.2]{BDG}, the symmetric gradient $\varepsilon(u):=(\frac{1}{2}(\partial_{i}u_{j}+\partial_{j}u_{i}))_{ij}$ for maps $u\colon\R^{n}\to\R^{n}$ is $\C$-elliptic for $n\geq 2$, and so is the trace-free gradient $\varepsilon^{D}(u):=\varepsilon(u)-\frac{1}{n}\mathrm{div}(u)E_{n}$ with the $(n\times n)$-unit matrix $E_{n}$ provided $n\geq 3$. These operators play a crucial role in elasticity or fluid mechanics; see, e.g., \cite{FuchsSeregin}.
	\end{example}

	\begin{example}[The $\curl$- and $\curl\curl$-operator]\label{ex:curl}
		Given $n\geq 2$ and $u\colon\R^{n}\to\R^{N\times n}$, we define $\mathrm{curl}(u)$ as in the introduction. Note that, for $v\in\C^{N\times n}$, $\curl[\xi](v)=0$ for $\xi\in\C^{n}\setminus\{0\}$ if and only if $v=a\xi^{\top}$ for some $a\in\C^{N}$, so $\dim_{\C}(\ker(\curl[\xi]))=N$. Similarly, for the Saint-Venant-compatibility complex, one explicitely verifies that $\curl\curl[\xi](v)=0$ if and only if $v=a\odot\xi=\frac{1}{2}(a\xi^{\top}+\xi a^{\top})$ for some $a\in\C^{n}$. Thus $\dim_{\C}(\ker(\curl[\xi]))=n$, and the validity of the constant rank property follows.
	\end{example}
	\begin{example}[Divergence-type operators]
		For $n\geq 2$ and $u=(u_{1},...,u_{n})\colon\R^{n}\to\R^{n}$, the divergence $\mathrm{div}(u)=\sum_{i=1}^{n}\partial_{i}u_{i}$ has symbol $\mathrm{div}[\xi](v):=\sum_{i=1}^{n}\xi_{i}v_{i}$. Therefore, with $\xi\in\C^{n}\setminus\{0\}$, we have $\sum_{i=1}^{n}\xi_{i}v_{i}=0$ provided $v\in {\xi}^{\bot}$, and thus $\dim_{\mathbb{C}}(\ker(\mathrm{div}[\xi]))=n-1$. Hence $\mathrm{div}$ is of constant complex rank. An operator that arises in the relaxation of static problems, cf.~\cite{CMO20}, is the divergence of symmetric matrices; the same argument as above establishes that the divergence of symmetric matrices is of constant complex rank.  
	\end{example}
	\begin{example}[The Laplacian]\label{ex:laplace}
The (scalar) Laplacian $\mathbb{B}=\Delta$ does not satisfy the constant rank condition over $\mathbb{C}$. For instance, let  $n=2$. Writing $\xi=(\xi_{1},\xi_{2})^{\top}\in\mathbb{C}^{2}$, the relevant symbol in view of Definition~\ref{def:CCR} is 
\begin{align}\label{eq:LaplacesymbolC}
\begin{split}
\mathbb{B}[\xi]=\xi^{\top}\xi & = \xi_1^2 + \xi_2^2.
\end{split}
\end{align} 
The polynomial given by~\eqref{eq:LaplacesymbolC} vanishes if and only if $\xi\in\mathbb{C}(1,\imag)^{\top}$ or $\xi\in\mathbb{C}(1,-\imag)^{\top})$, and so 
 \[
		\ker_{\mathbb{C}}(\B[\xi]) = \begin{cases} \mathbb{C} & \text{if}\; \xi = \lambda (1 ,\imag)^{\top} \text{ or } \xi = \lambda (1,-\imag)^{\top}, ~\lambda \in \mathbb{C} \\	\{0\} &\text{otherwise,} \end{cases} 
		\]
and so the constant rank condition is violated over $\mathbb{C}$; still, over the base field $\R$ the Laplacian is elliptic and hence of constant rank over $\R$. 
	\end{example}
	
	\section{A Nullstellensatz for operators of constant complex rank}\label{sec:nullstellensatz}
	The proof of Theorem~\ref{thm:main} hinges on a variant of the Hilbert Nullstellensatz from algebraic geometry stated in Theorem~\ref{thm:main1} below. For the reader's convenience, let us first display a classical version of the Hilbert Nullstellensatz as a background tool, which may e.g. be found in~\cite[\S 4.1, Thm.~2]{Cox}:
	\begin{lemma}[HNS]\label{lem:HNS}
		Let $\mathbb{F}$ be an algebraically closed field and $\mathfrak{A}\unlhd \mathbb{F}[X_{1},...,X_{n}]$ an ideal. Then we have $\sqrt{\mathfrak{A}}=\mathcal{I}(V(\mathfrak{A}))$, where 
		\begin{itemize}
			\item[-] $\sqrt{\mathfrak{A}}:=\{\mathbf{x}\in\mathbb{F}[X_{1},...,X_{n}]\colon\;\exists m\in\mathbb{N}_{0}\colon\;\mathbf{x}^{m}\in\mathfrak{A}\}$ is the \emph{radical of $\mathfrak{A}$}, 
			\item[-] $V(\mathfrak{A}):=\{x=(x_{1},...,x_{n})\in\mathbb{F}^{n}\colon\;\forall \mathbf{x}\in\mathfrak{A}\colon\;\mathbf{x}(x)=0\}$ is the \emph{set of common zeros of $\mathfrak{A}$}, and 
			\item[-] $\mathcal{I}(V(\mathfrak{A})):=\{\mathbf{x}\in\mathbb{F}[X_{1},...,X_{n}]\colon\;\forall x\in V(\mathfrak{A})\colon\;\mathbf{x}(x)=0\}$.
		\end{itemize}
	\end{lemma}
	The standard use of this result in the context of differential operators (see Remark~\ref{rem:insuff} below) does not prove sufficient for Theorem~\ref{thm:main}. Hence let $d,k,l\in\mathbb{N}$. For $i\in\{1,...,d\}$ and $j\in\{1,...,l\}$ we consider homogeneous polynomials $p_{ij} \in \C[\xi_1,...,\xi_n]$ of order $k$ and the system of equations 
	\begin{equation}\label{def:system}
	\sum_{i=1}^d p_{ij}(\xi)v_i =0, \qquad\xi=(\xi_{1},...,\xi_{n})\in\mathbb{C}^{n},\; j\in\{1,...,l\}, 
	\end{equation}
	where $v=(v_{1},...,v_{d})\in\mathbb{C}^{d}$. In accordance with Definition~\ref{def:CCR}, we say that the system \eqref{def:system} satisfies the \emph{constant rank property over $\C$} if there exists an $r \in \{0,...,d\}$ such that for every $\xi \in \C^n \setminus \{0\}$ the vector space \[
	\mathcal{X}_{\xi}((p_{ij})_{ij}):=\left\{ v=(v_{1},...,v_{d})\in \C^d \colon \sum_{i=1}^d p_{ij}(\xi) v_i = 0 ~ \text{for all}\; j\in\{1,...,l\} \right\}
	\]
	has dimension $(d-r)$ over $\C$. We may now state the main ingredient for the proof of Theorem~\ref{thm:main}, which arises as a generalisation of the usual Hilbert Nullstellensatz:
	\begin{theorem}[Vectorial Nullstellensatz for constant rank operators] \label{thm:main1}
		Let $d,k,l\in\mathbb{N}$ and, for $i\in\{1,...,d\}$ and  $j\in\{1,...,l\}$, $p_{ij} \in  \C[\xi_1,...,\xi_n]$ be homogeneous polynomials of degree $k$ such that ~\eqref{def:system} satisfies the \emph{constant rank property over $\C$}. Let $b_1,...,b_d \in \C[\xi_1,...,\xi_n]$, $v=(v_{1},...,v_{d})\in\mathbb{C}^{d}$ and define \[
		B[\xi] (v):= \sum_{i=1}^{d} v_i b_i(\xi).
		\]
		Suppose that for any $\xi=(\xi_{1},...,\xi_{n})\in \C^n \backslash \{0\}$ and $v=(v_{1},...,v_{d})\in \C^d$ we have that 
		\begin{align}\label{eq:NSTSatz0}
		\Big(\sum_{i=1}^d p_{ij}(\xi) v_i = 0 \;\;\text{for all}\; j\in\{1,..,l\}\Big) \quad \Longrightarrow \quad B[\xi](v) =0,
		\end{align}
		and let $q \in \C[\xi_1,...,\xi_n]$ be a homogeneous polynomial of degree $\geq 1$. Then there exist polynomials $h_j \in \C[\xi_1,...,\xi_n]$, $j\in\{1,...,l\}$, and an $m \in \N$, such that for all $\xi\in\C^{n}$ and all $v\in\C^{d}$ there holds
		\begin{align}\label{eq:NSTSatz}
		q^{m}(\xi) B[\xi](v) = \sum_{j=1}^l h_j(\xi) \sum_{i=1}^d v_i p_{ij}(\xi).
		\end{align}
	\end{theorem}
	\begin{proof}
		Let the polynomials $p_{ij}$ satisfy the constant rank property for some fixed $r \in \{0,...,d\}$. We  define sets  \[
		\mathcal{J}= \{ J \subset \{1,...,l\} \colon \vert J \vert =r\}, \quad \mathcal{I} = \{ I \subset \{1,...,d\} \colon \vert I \vert=r\}. 
		\]
		For a subset $J \in \mathcal{J}$ we write $J = \{j(1),...,j(r)\}$ for $j(1)<...<j(r)$ and likewise for $I \in \mathcal{I}$, $I=\{i(1),...,i(r)\}$ for $i(1) <...<i(r)$. Define the matrix $M_{IJ}\in\C^{r\times r}$ by its entries via \[
		(M_{IJ})_{\beta \gamma} := p_{i(\beta),j(\gamma)}.
		\]
		Now consider an arbitrary $(r\times r)$-minor of $P(\xi)=(p_{ij}(\xi))_{ij}$; any such minor arises as $\det(M_{IJ}(\xi))$ for some $I\in\mathcal{I},J\in\mathcal{J}$. If $\xi\in\C^{n}\setminus\{0\}$ is a common zero of all $q_{IJ}:=\det(M_{IJ})$, then $\dim_{\C}(\mathcal{X}_{\xi}((p_{ij})_{ij}))\neq d-r$ by virtue of the constant rank property over $\C$. On the other hand, by homogeneity of the $p_{ij}$'s, $\xi=0$ is a common zero of the $q_{IJ}$'s, and so is the only common zero of the $q_{IJ}$'s. 
		
		On the other hand, $\xi=0$ is a zero of any homogeneous polynomial $q\in\C[\xi_{1},...,\xi_{n}]$ of degree $\geq 1$. Thus, the Hilbert Nullstellensatz from Lemma~\ref{lem:HNS} implies the existence of an $m\in\mathbb{N}$ and polynomials $g_{IJ}\in\C[\xi_{1},...,\xi_{n}]$ ($I\in\mathcal{I},J\in\mathcal{J}$) such that \begin{equation}\label{eq:terminator}
		q^{m}= \sum_{J \in \mathcal{J}} \sum_{I \in \mathcal{I}} g_{IJ} \det(M_{IJ}).
		\end{equation}
		We now come to the definition of $h_j$ as appearing in \eqref{eq:NSTSatz}. For the matrix $M_{IJ}$ and $\gamma \in \{1,...,r\}$, we define the matrix $M_{IJ}^{\gamma}$ as the matrix where the $\gamma$-th column vector is replaced by $(b_{i(\beta)})_{\beta=1,...,r}$, i.e.,
		\[
		M_{IJ}^{\gamma} = \left( \begin{array}{ccccccc} p_{i(1)j(1)} &... &p_{i(1)j(\gamma-1)}& b_{i(1)} & p_{i(1)j(\gamma +1)} &...&p_{i(1)j(r)} \\
		...& &...&...&...& &... \\
		p_{i(r)j(1)} &... &p_{i(r)j(\gamma-1)}& b_{i(r)} & p_{i(r)j(\gamma +1)} &...&p_{i(r)j(r)} \end{array} \right). \]
		We then define for $j\in\{1,...,l\}$
		\begin{align}\label{eq:hjdef}
		h_j := \sum_{\gamma=1}^r \sum_{I \in \mathcal{I}} \sum_{J \in \mathcal{J} \colon j(\gamma)=j} g_{IJ} \det(M_{IJ}^{\gamma})
		\end{align}
		and claim that 
		\begin{align} \label{eq:Claim1}
		& \sum_{\gamma=1}^r p_{i j(\gamma)} \det(M_{IJ}^{\gamma}) = b_i \det M_{IJ}\qquad\text{for all}\;i\in\{1,...,d\}, \\ 
		&   \sum_{j=1}^l h_j \left( \sum_{i=1}^d p_{ij} v_i \right) = q^m \sum_{i=1}^d b_i v_i,  \label{eq:Claim2}
		\end{align}
		so that the $h_{j}$'s will satisfy \eqref{eq:NSTSatz}. Let us see how~\eqref{eq:Claim2}  follows from~\eqref{eq:Claim1}: In fact, \begin{align*}
		\sum_{j=1}^l h_j \left( \sum_{i=1}^d p_{ij} v_i \right) &\stackrel{\eqref{eq:hjdef}}{=} \sum_{j=1}^l \sum_{i=1}^d \sum_{\gamma=1}^r \sum_{I \in \mathcal{I}} \sum_{J \in \mathcal{J} \colon j(\gamma)=j}
		g_{IJ} \det(M_{IJ}^{\gamma}) p_{ij} v_i \\
		&\;= \sum_{J \in \mathcal{J}} \sum_{I \in \mathcal{I}} g_{IJ} \left( \sum_{i=1}^d \sum_{\gamma=1}^r p_{ij(\gamma)} \det(M_{IJ}^{\gamma}) v_i \right) \\
		&\stackrel{\eqref{eq:Claim1}}{=}\sum_{J \in \mathcal{J}} \sum_{I \in \mathcal{I}} g_{IJ} \det(M_{IJ}) \cdot \left( \sum_{i=1}^d b_i v_i \right) \\
		&\stackrel{\eqref{eq:terminator}}{=} q^m  \sum_{i=1}^d b_i v_i.
		\end{align*}
		Hence it remains to show~\eqref{eq:Claim1}. To this end, for $\beta,\gamma \in\{1,...,r\}$ let us define the matrix $M_{I(\beta)J(\gamma)}$ as the $(r-1) \times (r-1)$ matrix, where the $\gamma$-th  column of $M_{IJ}$ and the $\beta$-th row have been removed. By the Laplace expansion formula and the definition of $M_{IJ}^{\gamma}$, we then obtain
		\[
		\det(M_{IJ}^{\gamma}) = \sum_{\beta=1}^r (-1)^{\beta+\gamma} b_{i(\beta)} \det(M_{I(\beta)J(\gamma)}). \]
		Hence, 
		\begin{equation} \label{eq:zwischen}
		\sum_{\gamma=1}^r p_{i j(\gamma)} \det(M_{IJ}^{\gamma}) = \sum_{\beta,\gamma=1}^r (-1)^{\beta+\gamma} b_{i(\beta)} \det(M_{I(\beta)J({\gamma})}) p_{i j(\gamma)}.
		\end{equation}
		Now consider the $(r+1) \times (r+1)$-matrix $M$ defined by \[
		M := \left( \begin{array}{cccc}
		p_{i(1)j(1)} &\hdots& p_{i(1)j(r)} & b_{i(1)} \\
		\vdots & \ddots& \vdots & \vdots \\
		p_{i(r)j(1)} & ... & p_{i(r)j(r)} & b_{i(r)} \\
		p_{i j(1)} & ... & p_{i j(r)} & b_{i} \end{array} \right).
		\]
		By \eqref{eq:NSTSatz0}, for each $\xi\in\mathbb{C}^{n}\setminus\{0\}$ the subspace of $v \in \C^d$ such that \[
		\sum_{i=1}^d p_{ij}(\xi)v_{i} =0~\text{for all}\; j\in\{1,...,l\},\quad\quad\sum_{i=1}^d v_i b_i (\xi) =0 \] 
			is $\mathcal{X}_{\xi}((p_{ij})_{ij})$ and thus has dimension $(d-r)$. Therefore, all $(r+1)\times(r+1)$ minors of the matrix corresponding to these linear equations vanish. In particular, the determinant of the matrix $M$ is $0$. Denote by $M^{\beta}$ the $(r \times r)$-submatrix of $M$,  where the last column and the $\beta$-th row of $M$ are eliminated. We apply the Laplace expansion formula twice to $M$ (in the last column and then in the last row), to see that \begin{align*}
		0 &= \det(M) \\
		&= \left( \sum_{\beta=1}^r b_{i(\beta)} (-1)^{r+1+\beta} \det (M^{\beta}) \right) + b_i \det(M_{IJ}) \\
		&= \left( \sum_{\gamma=1}^r \sum_{\beta=1}^r (-1)^{r+1+\beta} (-1)^{r+\gamma} b_{i(\beta)} p_{ij(\gamma)} \det(M_{I(\beta)J(\gamma)}) \right)+ b_i \det(M_{IJ}).
		\end{align*}
		Therefore, \[
		b_i \det(M_{IJ}) = \sum_{\gamma=1}^r \sum_{\beta=1}^r(-1)^{\beta+\gamma} b_{i(\beta)} p_{ij(\gamma)} \det(M_{I(\beta)J(\gamma)}),
		\]
		which establishes \eqref{eq:Claim1}. The proof is complete. 
	\end{proof}
	\begin{remark}\label{rem:insuff}
		In the context of differential operators, the Hilbert Nullstellensatz is typically applied to $\mathbb{C}$-elliptic differential operators $\mathbb{A}$ as follows (cf.~\cite{Smith}, \cite[Lem.~4, Thm.~5]{Kala1}, \cite[Prop.~3.2]{GRV21}): Let $V\cong\R^{N}$, $W\cong\R^{m}$ and $\mathbb{A}$ be a first order differential operator on $\R^{n}$ from $V$ to $W$. Then $\mathbb{C}$-ellipticity of $\mathbb{A}$ implies by virtue of the Hilbert Nullstellensatz that there exists $k\in\mathbb{N}$ with the following property: There exists a linear, homogeneous differential operator $\mathbb{L}$ on $\R^{n}$ from $W$ to $V\odot^{k}\R^{n}$ of order $(k-1)$  such that $D^{k}=\mathbb{L}\mathbb{A}$. Inserting this relation into the usual Sobolev integral representation of $u\in\hold^{\infty}(\overline{\ball_{1}(0)};V)$ (cf.~\cite[\S 4]{Adams} or \cite[Thm.~1.1.10.1]{Mazya}) and integrating by parts then yields a polynomial $P$ of order $(k-1)$ such that 
		\begin{align*}
		u(x)=P(x) + \int_{\ball_{1}(0)}K(x,y)\mathbb{A}u(y)\dif y
		\end{align*}
		for all $x\in\ball_{1}(0)$ and all $u\in\hold^{\infty}(\overline{\ball_{1}(0)};V)$; here, the function $K\colon \ball_{1}(0)\times\ball_{1}(0)\to \mathscr{L}(W;V)$ is a suitable integral kernel. This, in particular, implies that $\dim(\ker(\mathbb{A}))<\infty$. 
		In our situation, a similar approach does not work. This is so because the operators $\mathbb{B},\widetilde{\mathbb{B}}$ from Theorem~\ref{thm:main} do not have finite dimensional nullspaces themselves; we may only assert that the nullspaces differ by finite dimensional vector spaces, and this is why we require the refinement provided by Theorem~\ref{thm:main1}.
	\end{remark}

	\section{Proof of Theorem~\ref{thm:main}}\label{sec:main}
	Based on Theorem~\ref{thm:main1}, the proof of  Theorem~\ref{thm:main} requires two additional ingredients that we record next:
	\begin{lemma} \label{aux}
		Let $\A \colon \hold^{\infty}(\R^n;\R^d) \to \hold^{\infty}(\R^n ;\R^l)$ be a homogeneous differential operator of order $k$. Define the differential operator 
\begin{align*}		
		\nabla \circ \A \colon \hold^{\infty}(\R^n;\R^d) \to \hold^{\infty}(\R^n ;\R^l \times \R^n)
\end{align*}		
		 componentwisely by \[
		((\nabla \circ \A) u)_i = \partial_i \A u,\qquad i\in\{1,...,n\}.
		\]
		Then we have
		\begin{align}\label{eq:kernablarep}
		\ker (\nabla \circ \A) = \ker(\A) + \mathscr{P}_{k}(\R^{n};\R^{d}). 
		\end{align}
	\end{lemma}
Observe that this result \emph{does not} require the constant rank property.
	\begin{proof}
		Suppose that $u \in \ker(\nabla \circ \A)$. Then $\A u$ is a constant function. Consider the space $W \subset \R^l$ defined by $W := \spano\{ \A[\xi](\R^{d})\colon \xi\in\R^{n}\}$. Note that, on the one hand, $\A u \in W$ pointwisely, and, on the other hand, 
		\begin{align}\label{eq:Wdef}
		\begin{split}
		W & = \A\mathscr{P}_{k}^{h}(\R^{n};\R^{d})=\A\mathscr{P}_{k}(\R^{n};\R^{d}).
		\end{split}
		\end{align}
		The last line can be seen by considering, for $|\beta|=k$ and $v\in\R^{d}$, the polynomials $p_{\beta}(x):=\tfrac{x^{\beta}}{\beta!}v$. Then, for any $\xi\in\R^{n}$, 
		\begin{align*}
		\A\Big(\sum_{|\beta|=k}\xi^{\beta}p_{\beta}\Big) = \sum_{|\alpha|=k}\sum_{|\beta|=k}\xi^{\beta}\A_{\alpha}\partial^{\alpha}p_{\beta} = \sum_{|\alpha|=k}\xi^{\alpha}\A_{\alpha}v
		\end{align*}
		and so \eqref{eq:Wdef} follows by the homogeneity of $\A$ of degree $k$. In particular, for every $u \in \ker(\nabla \circ \A)$, we can find a polynomial $p$ of degree $k$ with $\A (u-p)=0$. Hence $\ker(\nabla\circ\A)\subset \ker(\A)+\mathscr{P}_{k}(\R^{n};\R^{d})$. On the other hand, since $\A$ is homogeneous and of order $k$, every element of $\ker(\A)+\mathscr{P}_{k}(\R^{n};\R^{d})$ belongs to the nullspace of $\nabla\circ\A$. Thus~\eqref{eq:kernablarep} follows and the proof is complete. 
	\end{proof}
	\begin{corollary}[Kernels of annihilators] \label{coro:kernels}
		Let $\A^{(1)}$ and $\A^{(2)}$ be two homogeneous  differential operators of order $k^{(1)}$ and $k^{(2)}$, which have constant rank over $\C$ and both act on $\hold^{\infty}(\R^n;\R^d)$. Moreover, suppose that their Fourier symbols satisfy
		\begin{align}\label{eq:nullspaceinc1}
		\ker(\A^{(1)}[\xi]) \subset \ker(\A^{(2)}[\xi])\qquad\text{for all}\;\xi\in\C^{n}.
		\end{align}
		Then the following hold:
		\begin{enumerate}
			\item\label{item:inc1} There exists $\tilde{k} \in \N$ and a differential operator $\mathcal{B}$, such that \[
			\nabla^{\tilde{k}} \circ \A^{(2)} = \mathcal{B} \circ \A^{(1)}.
			\]
			\item\label{item:inc2}  For the nullspace of $\A^{(1)}$ we have \[
			\{u \in \lebe^1_{\locc} \colon \A^{(1)} u= 0 \} \subset \{ u \in \lebe^1_{\locc} \colon \A^{(2)} u =0 \}  + V,
			\]
			where $V$ is a finite dimensional vector space (consisting of polynomials).
			\item\label{item:inc3}
			If, in addition, \[ \ker(\A^{(1)}[\xi]) = \ker(\A^{(2)}[\xi]),
			\] then we may write  \[
			\{u \in \lebe^1_{\locc} \colon \A^{(1)} u= 0 \} + V = \{ u \in \lebe^1_{\locc} \colon \A^{(2)} u =0 \}  + W
			\]
			for finite dimensional vector spaces $V$ and $W$ consisting of polynomials.
			
		\end{enumerate}
	\end{corollary}
	
	\begin{proof}
		Ad~\ref{item:inc1}. We aim to apply Theorem~\ref{thm:main1}, and we explain the setting first. Assuming that $\A^{(1)}$ is $\R^{l_{1}}$-valued and $\A^{(2)}$ is $\R^{l_{2}}$-valued, we may write for $v=(v_{1},...,v_{d})\in\mathbb{C}^{d}$
\begin{align*}
\A^{(1)}[\xi]v = \Big(\sum_{i=1}^{d}A_{ij}^{(1)}(\xi)v_{i} \Big)_{j=1,...,l_{1}}\;\;\;\text{and}\;\;\; \A^{(2)}[\xi]v = \Big(\mathbb{A}_{m}^{(2)}(\xi)v\Big)_{m=1,...,l_{2}}, 
\end{align*}
where every $\mathbb{A}_{m}^{(2)}(\xi)v$ can be written as 
\begin{align*}
\mathbb{A}_{m}^{(2)}(\xi)v=\sum_{i=1}^{d}v_{i}b_{im}(\xi).
\end{align*}
For each $m\in\{1,...,l_{2}\}$, we apply Theorem~\ref{thm:main1} to $p_{ij}[\xi]=A_{ij}^{(1)}[\xi]$ and $B[\xi]=\mathbb{A}_{m}^{(2)}(\xi)$; note that its applicability is ensured by~\eqref{eq:nullspaceinc1}. 

In consequence, for every component $\A^{(2)}_m$ with $m\in\{1,...,l_{2}\}$ and $a\in\{1,...,n\}$, we may find $N(a,m) \in \N$ and polynomials $h_{j,a}\in \C[\xi_1,...,\xi_n]$, such that \[
		\xi_a^{N(a,m)} \A^{(2)}_m(\xi) =  \sum_{j=1}^{l_1} h_{j,a}(\xi) \sum_{i=1}^d A^{(1)}_{ij}(\xi) v_i.
		\]
		Therefore, choosing $\widetilde{k} := n \max_{m\in\{1,...,l_2\},a\in\{1,...,n\}} N(a,m)$, we obtain that for every $\alpha \in \N^n$ with $\vert \alpha \vert= \tilde{k}$ and $m\in\{1,...,l_2\}$, there exists $h_{j\alpha }$ such that \[
		\xi^\alpha \A^{(2)}_m(\xi) = \sum_{j=1}^{l_1} h_{j \alpha}(\xi) \sum_{i=1}^d A^{(1)}_{ij}(\xi) v_i.
		\]
		Defining the differential operator $\mathcal{B}$ according to this Fourier symbol, \ref{item:inc1} follows, i.e.,
\begin{align*}
\mathcal{B}[\xi] _{m,\alpha}  (w) = \sum_{j=1}^{l_1} h_{j \alpha}(\xi)w_{j}, \quad m\in\{1,\ldots,l_2\}.
\end{align*}
Ad~\ref{item:inc2}. This directly follows from Lemma~ \ref{aux}. Indeed, applying Lemma~\ref{aux} $\widetilde{k}$-times, there exists a finite dimensional space $\widetilde{V}$ of polynomials such that \[
		\{ u \in \lebe^1_{\locc}\colon \nabla^{\widetilde{k}}\A^{(2)}u =0 \} = \{u \in \lebe^1_{\locc} \colon \A^{(2)} u =0 \} + \widetilde{V}.
		\]
		As $\ker \A^{(1)} \subset \ker \mathcal{B} \circ \A^{(1)} = \ker \nabla^{\tilde{k}} \circ \A^{(2)}$, the result directly follows. Finally, \ref{item:inc3} is immediate by applying \ref{item:inc2} in both directions. The proof is complete.
	\end{proof}
	We may now turn to the
	\begin{proof}[Proof of Theorem~\ref{thm:main}]
		Direction~\ref{item:main1}$\Rightarrow$ \ref{item:main2} of Theorem~\ref{thm:main} is just Corollary~\ref{coro:kernels}; using convolution one may first observe this for $\lebe^1_{\mathrm{loc}}$ functions and then generalize it to $\mathscr{D}'$. On the other hand, direction \ref{item:main2}$\Rightarrow$\ref{item:main1} follows from a routine construction (see e.g. \cite{Smith,FoMu,GR19}) which we outline for the reader's convenience. Suppose towards a contradiction that there exists $\xi\in\mathbb{C}^{n}\setminus\{0\}$ such that $\ker(\B[\xi])\neq\ker(\widetilde{\B}[\xi])$. Without loss of generality, we may then assume there exists $v\in \C^{l}\setminus\{0\}$ such that $v\in\ker(\B[\xi])\setminus\ker(\widetilde{\B}[\xi])$. The proof is then concluded by considering the plane waves $u_h(x):=e^{\mathrm{i}x\cdot h\xi}v$ for $ h \in \mathbb{Z}$ and sorting by real and imaginary parts; passing to the span of $u_{h}$, $h\in\mathbb{Z}$, we obtain an infinite dimensional vector space which, up to the zero function, belongs to $\ker(\mathbb{B})\setminus\ker(\widetilde{\mathbb{B}})$.
	\end{proof}  
	\begin{example} \label{Bilaplace} In general, Theorem \ref{thm:main} will fail if $\B$ and $\widetilde{\B}$ \textit{do not} satisfy the complex constant rank property. As one readily verifies, if we take $\B = \Delta$ and $\widetilde{\B} = \Delta^2$ to be the Laplacian and the Bi-Laplacian (and so both violate the constant rank condition over $\mathbb{C}$ by Example~\ref{ex:laplace}) in $n=2$ dimensions,  \[
		\ker_{\mathbb{C}}(\B[\xi]) = \ker_{\mathbb{C}}(\widetilde{\B}[\xi]) = \left\{\begin{array}{ll} \mathbb{C} & \text{if}\; \xi = \lambda (1 ,\imag)^{\top} \text{ or } \xi = \lambda (1,-\imag)^{\top}, ~\lambda \in \mathbb{C} \\	\{0\} &\text{otherwise.} \end{array} \right.
		\]
Denote $\ker(\Delta)$ and $\ker(\Delta^{2})$ the nullspaces of $\Delta$ or $\Delta^{2}$, respectively, in $\mathscr{D}'(\R^{n})$. Denoting the homogeneous harmonic polynomials on $\R^{n}$ by $\mathscr{P}_{\mathrm{ho}}(\R^{n})$, we have 
\begin{align*}
\ker(\Delta)+\widetilde{\mathscr{P}}\subset\ker(\Delta^{2}), 
\end{align*}
where $\widetilde{\mathscr{P}}=\{v\colon\;\Delta v = p\;\text{for some}\;p\in\mathscr{P}_{\mathrm{ho}}(\R^{n})\}$, and from here one sees that the nullspaces of $\mathbb{B}$ and $\widetilde{\mathbb{B}}$ differ by an infinite dimensional vector space. 		
	\end{example}
	\begin{remark} \label{rem:inhom}
Up to now, we assumed that the polynomials $p_{ij}$ are homogeneous polynomials of order $k$. This assumption is motivated by the fact that we deal with homogeneous differential operators. However, we can also define the \textit{constant rank property} when not all polynomials have the same order. In particular, for polynomials $p_{ij}$ as in \eqref{def:system} we may weaken the assumption to $p_{ij}$ having order $k_j \in \N$, and the statement of the vectorial Nullstellensatz still holds true.
	
For the corresponding differential operator, this includes the following setting. The operator $\B= (\B_0,...,\B_k)$ is componentwisely defined via homogeneous differential operators $\B_i \colon \hold^{\infty}(\R^n;V) \to \hold^{\infty}(\R^n;W_i)$ of order $i$ (for $i=0$ the operator $\mathbb{B}_0$ is similarly understood to be a linear map). In particular, $\B \colon \hold^{\infty}(\R^n;V) \to \hold^{\infty}(\R^n;W_0 \times \dots \times W_k)$. The constant rank property in this setting means that there exists $r\in\mathbb{N}$ such that \[
\bigcap_{i=0}^k \ker(\B_i [\xi]) = r, \quad \text{for all}\;\xi \in \C^n \setminus \{0\}.
\]
Observe that it is not required at all, that each homogeneous component satisfies the constant rank property itself, e.g. $\B u=(\partial_1 u, \partial_2^2 u)$. 

In view of Lemma \ref{aux} we can however also transform this setting into a fully homogeneous one, while only allowing an additional finite-dimensional nullspace. Indeed, the operator $\widetilde{\mathbb{B}}$ given by \[
\widetilde{\B} = (\nabla^k \circ \B_0, \nabla^{k-1} \circ \B_1, \dots, \B_k)
\]
is homogeneous of order $k$ and its nullspace only differs by a finite dimensional space from the nullspace of $\B$.
	\end{remark}
	
	\begin{remark}\label{rem:leadover}
		For now, we have seen that if $\B[\xi]$ and $\widetilde{\B}[\xi]$ have the same nullspace for all $\xi\in\C^{n}\setminus\{0\}$, then their nullspaces as differential operators only differ by finite dimensional spaces. Given the nullspaces $V(\xi)= \ker (\B[\xi])$ for some differential operator $\B$, it is thus natural to ask for a \emph{minimal} differential operator in the sense of nullspaces, i.e., such that if $\ker(\B_0[\xi]) = V(\xi)$ and $\ker(\widetilde{\B}[\xi]) = V(\xi)$ for each $\xi\neq 0$, then $\ker(\B_0) \subset \ker(\widetilde{\B})$. 

To this end, let us recall some algebraic facts about ideals. Let $w_1,...,w_d$ be a basis of $W$. For a constant coefficient differential operator $\A$ with complex Fourier symbol $\A[\xi]$ we define the set of annihilator polynomials $\mathcal{B}$  as all vector valued polynomials vanishing on $\A[\xi]$, i.e. \begin{align*}
		\mathcal{B} = \{ P(\xi_1,...,\xi_n) = \sum_{i=1}^d p_i(\xi) w_i \colon P(\xi_1,...,\xi_n) \circ \A[\xi] =0 \}
		\end{align*}
		This $\mathcal{B}$ generates an ideal $\tilde{\mathcal{B}}$ in $\C[\xi_1,...,\xi_n,w_1,...,w_d]$. As every ideal in the ring of polynomials is finitely generated, so is $\widetilde{\mathcal{B}}$. In particular, there exists a finite generator $\mathcal{B}_0$ consisting of polynomials in $\mathcal{B}$; these are linear in $w_1,...,w_d$. As a consequence, every $P \in \mathcal{B}$ can be written as
		\begin{align} \label{eq:product}
		P(\xi) = \sum_{ P_j \in \mathcal{B}_0} \alpha_j(\xi) P_j
		\end{align}
		for some polynomials $\alpha_j$. In particular, this set $\mathcal{B}_0$ can be identified with a differential operator $\B_0$, which is component-wise homogeneous (where we view differential operators of degree zero as homogeneous of degree zero). Due to \eqref{eq:product} every differential operator $\B$ which is an annihilator of $\A$ can be written as \[
		\B = \mathbb{B}' \circ \B_0,
		\]
		hence $\ker(\B_0) \subset \ker(\B)$. Thus we might consider $\B_0$ as the \emph{natural} annihilator of $\A$.
	\end{remark}
	\section{A Poincar\'e-type lemma in $n=2$ dimensions}\label{sec:application}
In this concluding section we give a sample application of the results provided so far by proving a Poincar\'e lemma in two dimensions. For simplicity, we focus on first order operators and functions defined on a cube $Q=(0,1)^n$. For $\B$-free functions on the torus $\mathbb{T}_{n}$, it is well-known that $\A$, if $\A$ is a potential in the algebraic sense, it is also a potential in the sense that \[
	u \in \lebe^2(\mathbb{T}_n;W),\; \B u =0,\; (u)_{\mathbb{T}_{n}} =0 \quad \Longrightarrow \quad u=\A v \text{ for some } v \in \sobo^{1,2}(\mathbb{T}_{n};V).
	\]
	This is shown by use of Fourier methods. We cannot apply such a technique directly for functions on the cubes, as here boundary values cannot assumed to be periodic. Our strategy thus is to add a measure $\mu$ supported on $\partial Q$ such that for a function $u$ satisfying $\B u =0$ in $\mathrm{H}^{-1}(Q;W)$, the measure $u + \mu$ satisfies $\B(u+\mu)=0$ in $\mathrm{H}^{-2}(\mathbb{T}_n;W)$. We then can apply the theory on the torus to get some $v \in \lebe^2(\mathbb{T}_n;V)$ with $\A v= (u+\mu)$, i.e. $\A v = u$ in $Q$. In dimension $n=2$, we show that this strategy works for any differential operator of constant rank in $\C$ by adding measures on the one-dimensional faces of $Q$. In higher dimensions, there might be further restrictions on the operators, but e.g. for $\A=\curl$, $\B=\di$ one may show such a result by adding measures on one- and two-dimensional  faces.  
	
For the remainder of this section let $\A$ and $\B$ differential operators of first order given by
	\begin{align*}
	\A u = \sum_{k=1}^n A_k \partial_k u, \quad \B u= \sum_{k=1}^n B_k \partial_k u,
	\end{align*}
	where $A_k \in \Lin(\R^m;\R^d)$, $B_k \in \Lin(\R^d;\R^l)$. Let $Q=(0,1)^n$ and define \begin{align*}
	\lbq = \{ u \in \lebe^2(Q;\R^d) \colon \B u =0 \text{ in } \H^{-1}(Q;\R^{l})\}
	\end{align*} 
	and likewise \begin{align*}
	\H^1_{\B} (Q) = \{ u \in \H^1(Q;\R^d) \colon \B u =0 \;\text{in}\;\lebe^{2}(Q;\R^{l})\}, 
	\end{align*}
both being equipped with the usual norms on these spaces. For the following, we tacitly assume that $\mathbb{B}$ is an annihilator of $\mathbb{A}$ and that $\mathbb{B}$ has constant rank over $\mathbb{C}$. Our objective of this section is to establish the following result:
	\begin{theorem}\label{thm:firstorder}
		Let $n=2$. Then there exists a finite dimensional space $X \subset \hilb^1_{\B}(Q)$ consisting of polynomials and a linear, bounded map $\A^{-1}\colon \hilb^1_{\B}(Q) \cap \ker(\B) \to \lebe^2(Q;\R^m)$, such that $\A \circ \A^{-1} u - u \in X$. If, in addition, the operator $\B$ is spanning (cf.~Lemma~\ref{lem:Aux1} for this terminology), then $X =\{0\}$.
	\end{theorem}
In consequence, in the situation of Theorem~\ref{thm:firstorder} we may write $u=\A(\A^{-1}u)+\pi$ for some polynomial $\pi\in X$. We split the proof of Theorem~\ref{thm:firstorder} into several steps.
	\begin{lemma}
		Let $n=2$. We can decompose 
		\begin{align}\label{eq:DecompMain}
		\R^d= V_0 + V_1 + V_2,
		\end{align}
		such that $V_i \cap V_j =\{0\}$, $V_i \perp V_j$ for $i,j\in\{0,1,2\}$ with $i\neq j$ and \begin{align*}
		V_0&= \left(\spano_{\xi \in \R^2\setminus \{0\} } \ker (\B[\xi]) \right)^{\perp} = \left( \spano ( \ker (\B[e_1] )\cup (\ker \B[e_2])) \right)^{\bot},  \\
		V_2&= \bigcap_{\xi \in \R^2 \setminus \{0\}} \ker (\B[\xi]) = \ker(\B[e_1]) \cap \ker(\B[e_2]).
		\end{align*}
	\end{lemma}
	\begin{proof}
		Clearly, $V_0 \perp V_2$, so $V_1$ may be just chosen accordingly. It remains to show that $V_0$ and $V_2$ can be represented in terms of the behaviour of $\B[e_1]$ and $\B[e_2]$. As $\B$ is of order one, then $v \in \ker \B[e_1] \cap \ker \B[e_2]$ implies by linearity that $v \in \ker \B[\lambda e_1 + \mu e_2]$ for all $\lambda,\mu\in\R$, showing the characterisation of $V_2$. On the other hand, if $\A$ is of order one, then for all $\xi =\xi_1 e_1+ \xi_2 e_2 \in \R^2$ \begin{align*}
		\Image (\A[\xi_1 e_1 + \xi_2 e_2]) \subset \Image (\A[e_1]) + \Image(\A [e_2]).
		\end{align*}
		As $\Image(\A[\xi]) = \ker(\B[\xi])$, we get the desired result for $V_0$.
	\end{proof}
	For the following, observe that we may define another  differential operator \[\widetilde{\B} \colon \hold^{\infty}(\R^2;\R^d) \to \hold^{\infty}(\R^2;\R^l \times V_{0})\] by defining $\widetilde{\B}(u) = (\B u , P_{V_0}(u))$, where $P_{V_0}$ denotes the orthogonal projection onto $V_0$. Then $\ker(\widetilde{\B}[\xi]) = \ker (\B [\xi])$ for all $\xi \in \mathbb{C}^2 \setminus \{0\}$. In view of Remark~\ref{rem:inhom}, we have $\ker(\B) = \ker(\widetilde{\mathbb{B}}) + X$ for some finite dimensional subspace $X \subset \lbq$.  Note that $\widetilde{\B}$ is not homogeneous in total but in its single components; this will suffice for the following. 
	As a consequence, we may assume from now on that $V_0 =0$ by considering $\widetilde{\B}$ instead of $\B$. This is why we have the finite dimensional space $X$ in the formulation of Theorem \ref{thm:firstorder}.

	\begin{lemma}\label{lem:Aux1}
		Suppose that $\B$ is \emph{spanning}, i.e., $V_0=\{0\}$ in \eqref{eq:DecompMain} and that the union $\bigcup_{\xi \in \R^2} \Image(\B[\xi])$ spans $\R^l$. Then we have \[
		\R^l = \spano_{\xi \in \R^2 \setminus \{0\}} \Image (\B[\xi]) = \Image(\B[e_1]) = \Image(\B[e_2]) =\Image(\B[\xi])
		\]
for all $\xi\in\R^{2}\setminus\{0\}$. 
	\end{lemma}
Let us shortly remark that for the kernel of the differential operator $\B$, we might restrict our study to operators, such that  \[
	\R^l = \spano_{\xi \in \R^2 \setminus \{0\}} \Image (\B[\xi])
	\]
	If this is not satisfied, we might define the vector space $Y$ as above span and consider $\B'=P_Y \circ \B$, where $P_Y$ is the orthogonal projection onto $Y$. Then $\ker \B'=\ker \B$ and $\B'$ satisfies \[
	\spano_{\xi \in \R^2 \setminus \{0\}} \Image (\B'[\xi]) =Y,
	\]
	i.e. satisfies the assertions of Lemma \ref{lem:Aux1}.
	\begin{proof}[Proof of Lemma \ref{lem:Aux1}]
		Suppose there exist $\xi_1,\xi_2 \in \R^2\setminus\{0\}$ such that $\Image(\B[\xi_1]) \neq \Image(\B[\xi_2])$. In particular, $\xi_1$ and $\xi_2$ are linearly independent. Moreover, $\ker(\B^{\ast}[\xi_1]) \neq \ker(\B^{\ast}[\xi_2])$ and so there exists some $w \in \R^l$ such that $w \in \ker \B^{\ast}[\xi_2]$ but $w \notin \ker \B^{\ast}[\xi_1]$. 
		Therefore $0\neq v:= \B^{\ast}[\xi_1 + \lambda \xi_2]w \in \Image(\B^{\ast}[\xi_1 + \lambda \xi_2])$ for any $\lambda \in\R$. As $\Image(\B^{\ast}[\xi]) = \left( \ker \B[\xi] \right)^{\perp}$, $P_{\ker(\B[\xi_{1}+\lambda\xi_{2}])}(v) =0$, where again $P_V$ denotes the orthogonal projection onto the subspace $V \subset \R^d$. The map 
		\begin{align}\label{eq:muellermappo}
		\xi \mapsto P_{\ker(\B[\xi])}(\cdot)
		\end{align}
		is homogeneous of degree zero and continuous for $\B$ satisfying the constant rank property \cite[Prop.~2.7]{FoMu}. Every $\xi\in\R^{2}\setminus\R\xi_{2}$ can be written as $\xi=\mu(\xi_{1}+\lambda\xi_{2})$ for suitable $\lambda\in\R$ and $\mu\in\R\setminus\{0\}$. For such $\xi$, the zero homogeneity of~\eqref{eq:muellermappo} yields $P_{\ker(\B[\xi])}(v)=0$. On the other hand, choosing $\mu=\lambda^{-1}$ and letting $\lambda \to \infty$, the continuity of~\eqref{eq:muellermappo} we conclude that $P_{\ker(\B[\xi])}(v)=0$ for any $\xi\in\R^{2}\setminus\{0\}$. Combining this with the zero homogeneity of~\eqref{eq:muellermappo}, we also obtain $v\in (\ker(\mathbb{B}[\theta\xi_{2}]))^{\bot}$ for all $\theta\in\R\setminus\{0\}$. Hence, $v\in(\ker(\mathbb{B}[\xi]))^{\bot}$ for all $\xi\in\R^{2}\setminus\{0\}$, and this contradicts our assumption $V_{0}=\{0\}$. The proof is complete.
	\end{proof}
	\begin{lemma}\label{lem:Aux2}
		Let $\xi_1,\xi_2 \in \R^2$ be linearly  independent and $\B$ be spanning in the sense of Lemma~\ref{lem:Aux1}. Then there is a linear map $L_{\xi_1,\xi_2} \colon \R^l \to \ker (\B[\xi_1])$ with \begin{align*}
		\B[\xi_2] \circ L_{\xi_1,\xi_2} = \id_{\R^{l}}.
		\end{align*}
	\end{lemma}
	\begin{proof}
	For two finite dimensional real vector spaces $X_{1},X_{2}$, we first recall that a linear map $T\colon X_{1}\to X_{2}$ has a right inverse $S\colon X_{2}\to X_{1}$ if and only if $T$ is surjective. In view of the lemma, we thus have to establish that $\B[\xi_2]|_{\ker(\B[\xi_{1}])}$ is surjective, and this follows from a dimensional argument as follows: Let $r= \dim(V_2 )$ and $ s= \dim(\ker(\B[\xi]))$, which does not depend on $\xi \in \R^2 \setminus \{0\}$ due to the constant rank property. As $\B$ is spanning, 
		\[ d = \dim (\ker(\B[\xi_1])) + \dim (\ker (\B[\xi_2])) - \dim(\ker (\B[\xi_1]) \cap \ker (\B[\xi_2])) = 2s -r.
		\]
By Lemma~\ref{lem:Aux1}, $\R^l= \Image(\B[\xi_1])$, and thus the rank-nullity theorem yields \[
		l = \dim(\Image(\B[\xi_1])) = d - \dim (\ker (\B[\xi_1])) = (2s -r) - s = s-r.
		\]
		On the other hand, restricting $\B[\xi_2]$ to $\ker \B[\xi_1]$, the nullspace of $\B[\xi_{2}]|_{\ker(\mathbb{B}[\xi_{1}])}$ is $V_0$, hence its dimension is $r$, and the dimension of its image is $s-r$. Hence, $\B[\xi_2]$ restricted to $\ker \B[\xi_1]$ is still surjective onto $\R^l$, and therefore such a map $L_{\xi_{1},\xi_{2}}$ exists.
	\end{proof}
The second key ingredient to establish Theorem \ref{thm:firstorder} is the adding of measures on the boundary. In particular, we aim to add a measure $\mu$ such that $u+\mu$ is $\B$-free as a measure \emph{on the torus $\mathbb{T}_{2}$}: 
	\begin{lemma}[Adding measures on the boundary]\label{lem:addmeas}
		There are linear maps $S_1,S_2$ with the following properties:
		 \begin{enumerate}
			\item $S_1 \colon \hilb^1_{\B}(Q) \to \mathscr{P}_{2}(\R^{2};\R^{d})\cap\ker(\mathbb{B})$,
			\item $S_2 \colon \hilb^1_{\B}(Q) \to \lebe^2(\partial Q;\R^{d}) (\hookrightarrow \hilb^{-1}(Q;\R^{d}))$,
			\item $\B (u + S_1 u + S_2 u) = 0$ in $\hilb^{-2} (\mathbb{T}_2;\R^l)$ for all $u\in\mathrm{H}_{\mathbb{B}}^{1}(Q)$.
		\end{enumerate}
	\end{lemma}
	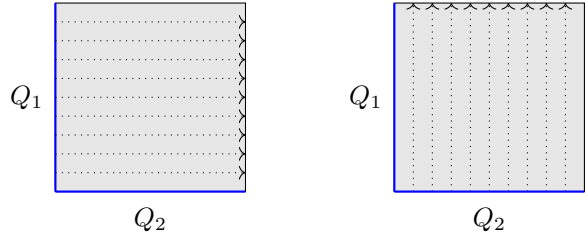
\begin{figure}
	\begin{tikzpicture}[scale=1.25]
	\draw[-,fill=black!10!white] (0,0)--(2,0)--(2,2)--(0,2)--(0,0);
	\node at (-0.3,1) {$Q_{1}$};
	\node at (1,-0.3) {$Q_{2}$};
	\draw[-,blue,thick] (0,0)--(0,2);
	\draw[dotted,black,->] (0,1.8)--(2,1.8);
	\draw[dotted,black,->] (0,1.6)--(2,1.6);
	\draw[dotted,black,->] (0,1.4)--(2,1.4);
	\draw[dotted,black,->] (0,1.2)--(2,1.2);
	\draw[dotted,black,->] (0,1.0)--(2,1.0);
	\draw[dotted,black,->] (0,0.8)--(2,0.8);
	\draw[dotted,black,->] (0,0.6)--(2,0.6);
	\draw[dotted,black,->] (0,0.4)--(2,0.4);
	\draw[dotted,black,->] (0,0.2)--(2,0.2);
	\draw[-,blue,thick] (0,0)--(2,0);
	\end{tikzpicture}
	\hspace{1cm}
	\begin{tikzpicture}[scale=1.25]
	\draw[-,fill=black!10!white] (0,0)--(2,0)--(2,2)--(0,2)--(0,0);
	\node at (-0.3,1) {$Q_{1}$};
	\node at (1,-0.3) {$Q_{2}$};
	\draw[-,blue,thick] (0,0)--(0,2);
	\draw[dotted,black,->] (1.8,0)--(1.8,2);
	\draw[dotted,black,->] (1.6,0)--(1.6,2);
	\draw[dotted,black,->] (1.4,0)--(1.4,2);
	\draw[dotted,black,->] (1.2,0)--(1.2,2);
	\draw[dotted,black,->] (1,0)--(1,2);
	\draw[dotted,black,->] (0.8,0)--(0.8,2);
	\draw[dotted,black,->] (0.6,0)--(0.6,2);
	\draw[dotted,black,->] (0.4,0)--(0.4,2);
	\draw[dotted,black,->] (0.2,0)--(0.2,2);
	\draw[-,blue,thick] (0,0)--(2,0);
	\end{tikzpicture}
	\caption{Cube notation and the idea in the proof of Lemma~\ref{lem:addmeas}. We \emph{periodify} the given functions to access the theory on the two-dimensional torus $\mathbb{T}_{2}$ in Proposition~\ref{prop:OliverKahn}. To enforce periodicity, the non-periodic contributions of some $u$ are handled by adding suitable correctors defined in terms of horizontal or vertical line integrals, respectively.}
	\end{figure}
	\begin{proof}
		Recall that the trace operator is bounded from $\mathrm{H}^1(Q;\R^d)$ to $\lebe^2(\partial Q; \R^d)$. Define $Q_1 =\{0\} \times [0,1]$
		and $Q_2 =[0,1] \times \{0\}$, which may both be seen as subsets of $Q$ and the torus $\mathbb{T}_2$. Define for $u\in\mathrm{H}_{\mathbb{B}}^{1}(Q;\R^{d})$
		\begin{align*}
		w_1(y) := \B[e_1] (u(0,y) - u(1,y))\quad\text{and}\quad w_2(x) :=\B[e_2] (u(x,0) - u(x,1))
		\end{align*}
for $\mathscr{L}^{1}$-a.e. $x,y\in [0,1]$. Then $u \mapsto w_j$ is linear and bounded from $\mathrm{H}^1_{\B}(Q) \to \lebe^2([0,1];\R^l)$. We then put 
\begin{align*}
c_{1}:=c_{1}(u):=\int_{0}^{1}w_{1}(y)\dif y\quad\text{and}\quad c_{2}:=c_{2}(u):=\int_{0}^{1}w_{2}(x)\dif x
\end{align*}
and observe that, because of $u\in\mathrm{H}_{\mathbb{B}}^{1}(Q)$ and a subsequent integration by parts, 
\begin{align}\label{eq:oasis}
0=\int_{Q}\mathbb{B}u\dif x = \int_{\partial Q}\mathbb{B}[\nu_{\partial Q}]u\dif\mathscr{H}^{1}
\end{align}
with the outer unit normal $\nu_{\partial Q}$ to $\partial Q$. Decomposing $\partial Q$ into its single faces and using the definition of $c_{1},c_{2}$, we find that $c_{1}=-c_{2}$. 

Now define the polynomial $S_1 u$ as follows: \begin{equation}
		    S_1 u(x_1,x_2) := a_{11} x_1^2 + 2 a_{12} x_1 x_2 + a_{22} x_2^2
		\end{equation}
	    for $a_{11},a_{12},a_{22} \in \R^d$ defined in terms of the maps $L$ from Lemma~\ref{lem:Aux2} via
	    \begin{align}\label{eq:stefdef}
		a_{12} :=- L_{e_1,e_2-e_1} (c_1), \quad a_{11} := L_{e_2,e_1} (-\B[e_2] a_{12}), \quad a_{22} := L_{e_1,e_2} (-\B[e_1] a_{12}).
		\end{align}
By the properties of the maps $L$ as displayed in Lemma~\ref{lem:Aux2}, we have \begin{align}\label{eq:auxisystem}
\begin{split}
		    \B[e_1] a_{11} + \B[e_2] a_{12} &= \B[e_1] (L_{e_2,e_1} (-\B[e_2] a_{12})) + \B[e_2] a_{12} = 0, \\
		    \B[e_2] a_{22} + \B[e_1] a_{12} &= \B[e_2]  (L_{e_1,e_2} (-\B[e_1] a_{12})) + \B[e_1] a_{12} =0.
		    \end{split}
		\end{align}
This particularly implies that 
\begin{align}\label{eq:Bvanish}
\begin{split}
\mathbb{B}S_{1}u & = \mathbb{B}[e_{1}]\partial_{1}S_{1}u + \mathbb{B}[e_{2}]\partial_{2}S_{1}u \\ 
& = \mathbb{B}[e_{1}](2a_{11}x_{1}+2a_{12}x_{2})+\mathbb{B}[e_{2}](2a_{12}x_{1}+2a_{22}x_{2})\stackrel{\eqref{eq:auxisystem}}{=}0. 
\end{split}
\end{align}
For future reference, we now record that 
\begin{align}\label{eq:antiprojection}
S_{1}(S_{1}u + u ) = 0\qquad\text{for all}\;u\in\mathrm{H}_{\mathbb{B}}^{1}(Q), 
\end{align}
which can be seen as follows: With the obvious definition of $\widetilde{c}_{1}$,  
\begin{align*}
\widetilde{c}_{1} := \int_{0}^{1}\widetilde{w}_{1}(y)\dif y & := \int_{0}^{1}\mathbb{B}[e_{1}](S_{1}u(0,y)-S_{1}u(1,y))\dif y \\ 
& = \int_{0}^{1}\mathbb{B}[e_{1}](2a_{22}y^{2}-a_{11}-2a_{12}y)\dif y \\
& = \int_{0}^{1}\mathbb{B}[e_{1}](-a_{11}-2a_{12}y)\dif y\;\;\;\;\;\;(\text{by~\eqref{eq:stefdef} and Lemma~\ref{lem:Aux2}})\\
& = -\mathbb{B}[e_{1}]a_{11}-\mathbb{B}[e_{1}]a_{12} \\ 
& = \mathbb{B}[e_{2}-e_{1}]a_{12} = -c_{1},
\end{align*}
the ultimate two equalities being valid by~\eqref{eq:stefdef} and Lemma~\ref{lem:Aux2} as well. Using that $\mathbb{B}S_{1}u=0$, we may argue as in~\eqref{eq:oasis}ff. to find that 
\begin{align*}
\widetilde{c}_{2} := \int_{0}^{1}\widetilde{w}_{2}(x)\dif y := \int_{0}^{1}\mathbb{B}[e_{2}](S_{1}u(x,0)-S_{1}u(x,1))\dif x = c_{1} = -c_{2}.
\end{align*}
This implies that $S_{1}(S_{1}u)=-S_{1}u$ and  hereafter~\eqref{eq:antiprojection}. 

We now come to the definition of $S_{2}u\colon Q_{1}\cup Q_{2}\to \R^{d}$. If $S_{1}u\equiv 0$, we then define 
\begin{align} \label{def:s2}
		S_2 u(0,y) := -\int_{0}^y L_{e_1,e_2} w_1(t) \dt, \quad S_2 u(x,0) := -\int_{0}^x L_{e_2,e_1} w_2(t) \dt.
		\end{align}
In general, we recall~\eqref{eq:antiprojection} and  define for general $u\in\mathrm{H}_{\mathbb{B}}^{1}(Q)$ 
\begin{align*}
S_{2}u:=S_{2}(u+S_{1}u).
\end{align*}
Then $S_2u$ defined on $Q_1 \cup Q_2$ has the following properties: \begin{enumerate}
			\item\label{item:meisterpropper1} $S_2u(0,0) = S_2u(0,1) = S_2u(1,0) =0$ due to $c_1=c_2 =0$. Indeed, since $u\in\mathrm{H}_{\mathbb{B}}^{1}(Q)$ satisfies $S_{1}u\equiv 0$, we conclude $a_{12}=0$. On the other hand, $L_{e_{1},e_{2}-e_{1}}$ is injective by Lemma~\ref{lem:Aux2} and so $c_{1}=0$ in light of~\eqref{eq:stefdef}; but then $c_{2}=-c_{1}=0$ as well.
			\item\label{item:meisterpropper2} $S_2u \in \lebe^2(Q_{1}\cup Q_{2};\R^{d})$.
			\item\label{item:meisterpropper3} $S_2u(0,\cdot) \in \ker(\B[e_1])$, $S_2u(\cdot,0) \in \ker(\B[e_2])$ by Lemma~\ref{lem:Aux2}.
			\item\label{item:meisterpropper4} $S_2u(0,\cdot), S_2u(\cdot,0) \in \mathrm{H}^1_0((0,1))$ and, again by Lemma~\ref{lem:Aux2},
			\begin{align*}        
			\B[e_2]\tfrac{\dif}{\dif t} S_2 u(0,t) = -w_1(t),\;\;\;\B[e_1] \tfrac{\dif}{\dif t}S_2 u(t,0) = -w_2(t).
			\end{align*}
		\end{enumerate}
By periodicity, we may view $S_2u \in \lebe^2(\partial Q;\R^{d})$, and this can be seen as an element of $\mathrm{H}^{-1}(\mathbb{T}_2;\R^d)$ by identifying it with the bounded linear functional 
		\begin{align*}
\mathrm{H}^{1}(\mathbb{T}_{2};\R^{d})\ni \psi\mapsto \int_{Q_{1}}S_{2}u\cdot\mathrm{tr}(\psi)\dHaus^{1} +\int_{Q_{2}}S_{2}u\cdot\mathrm{tr}(\psi)\dHaus^{1} .
		\end{align*}
Thus, for all $\phi \in \mathrm{H}^2(\mathbb{T}_2;\R^l)$ we have \begin{align*}
		\langle \B S_2 u &, \phi \rangle_{\mathrm{H}^{-2}(\mathbb{T}_{2})\times\mathrm{H}^{2}(\mathbb{T}_{2})} = - \int_{Q_1} S_2u \cdot \mathrm{tr}(\B^{\ast} \phi) \dHaus^1 - \int_{Q_2} S_2 u \cdot \mathrm{tr}(\B^{\ast} \phi )\dHaus^1\\
		&= - \int_{Q_1} (\B[e_1] S_2 u)\cdot \mathrm{tr}(\partial_1 \phi) + (\B[e_2] S_2 u )\cdot\mathrm{tr}(\partial_2 \phi) \dHaus^1 \\ & \;\; \;\;- \int_{Q_2} (\B[e_1] S_2 u)\cdot \mathrm{tr}(\partial_1 \phi) + (\B[e_2] S_2 u)\cdot\mathrm{tr}(\partial_2 \phi) \dHaus^1 \\
		&\!\stackrel{\ref{item:meisterpropper3}}{=} - \int_{Q_1} (\B[e_2] S_2 u)\cdot \mathrm{tr}(\partial_2 \phi)\dHaus^{1} - \int_{Q_2} (\B[e_1] S_2 u)\cdot\mathrm{tr}(\partial_1 \phi)\dHaus^1 \\
		& = \int_{Q_1} (\B[e_2] \partial_2 S_2 u)\cdot \mathrm{tr}(\phi)\dHaus^1 + \int_{Q_2}  (\B[e_1] \partial_1 S_2 u)\cdot\mathrm{tr}(\phi) \dHaus^1 \\
		&\!\stackrel{\ref{item:meisterpropper4}}{=} -\int_{Q_1} w_1 \cdot\mathrm{tr}(\phi) \dHaus^1- \int_{Q_2} w_2 \cdot\mathrm{tr}(\phi) \dHaus^1.
		\end{align*}
		On the other hand, for any $\phi \in \mathrm{H}^2(\mathbb{T}_2;\R^l)$
		\begin{align*}
		\langle \B u, \phi \rangle_{\mathrm{H}^{-2}(\mathbb{T}_{2})\times\mathrm{H}^{2}(\mathbb{T}_{2})}   &= - \int_{\mathbb{T}_2} u\cdot\B^{\ast} \phi \dx = \int_Q \B u \cdot\phi \dx- \int_{\partial Q} (\B[\nu_{\partial\Omega}] u)\cdot \mathrm{tr}(\phi) \dHaus^1 \\
		&= \int_{Q_1}  w_1 \cdot\mathrm{tr}(\phi) \dHaus^1+ \int_{Q_2} w_2 \cdot\mathrm{tr}(\phi) \dHaus^1.
		\end{align*}
Hence, $\mathbb{B}(u+S_{2}u)=0$ in $\mathrm{H}^{-1}(\mathbb{T}_{2};\R^{l})$ whenever $u\in\mathrm{H}_{\mathbb{B}}^{1}(Q)\cap\{S_{1}u\equiv 0\}$.  In the general case, we apply the foregoing result to $u+S_{1}u$ and hence obtain 
\begin{align*}
\mathbb{B}((u+S_{1}u) + S_{2}(u+S_{1}u))=0.
\end{align*}
To conclude, as $S_2 u = S_2 (u+S_1u)$, we have $\B(u+S_1 u +S_2 u) =0$ as an element of $\hilb^{-2}(\mathbb{T}_2;\R^l)$, and the proof is complete. 
	\end{proof}	
	
	\begin{proposition}\label{prop:OliverKahn}
		Suppose that $\B$ satisfies the spanning condition. There is a linear and bounded map $\A^{-1} \colon \hilb^1_{\B}(Q) \to \lebe^2(Q;\R^m)$, such that $\A \circ \A^{-1} = \id$, meaning that for all $u\in\mathrm{H}_{\mathbb{B}}^{1}(Q)$ and all $\phi \in \hilb^1_0(Q;\R^d)$ \begin{align*}
		\int_{Q} \A^{-1} u \cdot \A^{\ast} \phi = \int_{Q} u \phi.
		\end{align*}
	\end{proposition}
	\begin{proof}
Given $u\in\mathrm{H}_{\mathbb{B}}^{1}(Q)$, we write $u = (u+S_1 u)+(-S_1 u) =: u_{1} + u_{2}$ with $S_{1}$ as in the preceding lemma. We treat $u_{1}$ and $u_{2}$ separately.

Recall that $S_2 u_1 =S_2u$ for $S_2$ as in the previous lemma. We write \[
	u_1 + S_2 u_1 = u_0 + \bar{u}
	\]
 for some $u_{0}\in\R^{d}$ and $\bar{u}\in\mathrm{H}^{-1}(\torus_2;\R^d)$, where $\bar{u}$ has zero average over $\T_2$, i.e. \[
 \langle v, 1 \rangle_{\mathrm{H}^{-1}\times\mathrm{H}^1} =0.
 \]
 Note that $\B \bar{u} =0$ in $\mathrm{H}^{-2}(\torus_2;\R^l)$.
  By the same argument as in Lemma \ref{aux} we can write $u_{0}=\A P_{1}$ for a suitable polynomial $P_{1}$ of order one  with mean value zero; moreover, the map $u_{0}\mapsto P_{1}$ can be arranged to be linear.

For $\bar{u}$, we can apply the theory for constant rank operators on the torus. In particular, by \cite{raitanew} there exists a linear and bounded operator $\A_{\torus}^{-1} \colon \mathrm{H}^{-1}(\torus_2;\R^d) \to \lebe^2(\torus_2;\R^m)$ that satisfies  \begin{equation} \label{Ainv}
    \A \circ \A_{\torus}^{-1} v = v \quad\text{for all}\; v \in \mathrm{H}^{-1}(\torus_2;\R^d) \text{ with } \B v =0\;\text{and}\;\langle v, 1 \rangle_{\mathrm{H}^{-1}\times\mathrm{H}^1} =0.
    \end{equation}
	Thus, defining $w := P_1 +  \A_{\torus}^{-1} \bar{u}$, we conclude that \begin{enumerate}
	    \item $w$ depends linearly on $u$;
	    \item $\Vert w \Vert_{\lebe^2} \leq c(\Vert u_0 \Vert_{\lebe^2} + \Vert \bar{u} \Vert_{\mathrm{H}^{-1}}) \leq C \Vert u \Vert_{\mathrm{H}^1}$.
   \item $\A w = (u_{1} + S_2 u)$.
	\end{enumerate}
We now establish that $u_{2}$ can be written as $u_{2}=\A P_{2}$ for a third order polynomial $P_{2}$.	
Recall that \[
	S_1 u(x_1,x_2)= a_{11} x_1^2 + 2 a_{12} x_1 x_2 + a_{22} x_2^2
	\]
	with $a_{ij}$ defined as in \eqref{eq:stefdef}. We now define polynomials $P_3$ and $P_4$, such that $\A(P_3+P_4) = -S_1 u$.
	
\textit{Definition of $P_3$}: By the definition of the map $L$ from Lemma~\ref{lem:Aux2} and by \eqref{eq:stefdef}, the coefficients $a_{ij}$ obey the following: \[
	a_{11} \in \ker (\B[e_2]), \quad a_{22} \in \ker (\B[e_1]).
	\]
	The differential operator $\A$ is a potential of $\B$. Therefore, for any $\xi \in \R^2 \setminus \{0\}$, there is a linear map \[
	\A^{-1}[\xi] \colon \ker(\B [\xi]) \to (\ker \A [\xi])^{\perp}
 	\]
	with $\A[\xi] \circ \A^{-1}[\xi]=\mathrm{Id}_{\ker(\B[\xi])}$ (seen as a Fourier multiplier, this map exactly defines the operator in \eqref{Ainv}). For future reference, we note that expanding $\mathbb{B}[\xi]\A[\xi]=0$ for $\xi=\xi_{1}e_{1}+\xi_{2}e_{2}\in\R^{2}$ particularly yields 
	\begin{align}\label{eq:kloppo}
	\xi_{1}\xi_{2}(\mathbb{B}[e_{1}]\mathbb{A}[e_{2}]+\mathbb{B}[e_{2}]\mathbb{A}[e_{1}])=0. 
	\end{align} 
Let us define \[
	P_3 (x_{1},x_{2}) := -\A^{-1}[e_2] (a_{11}) x_1^2 x_2 - \A^{-1}[e_1](a_{22}) x_1 x_2^2.
	\]
	Observe that $(-S_1 u-\A P_3)$ still satisfies $\B (- S_1 u - \A P_3) =0$ by virtue of $\mathbb{B}[\xi]\A[\xi]=0$ and~\eqref{eq:Bvanish}, and has the form \begin{equation}\label{tuchel}
	\begin{split}
	& (- S_1 u - \A P_3) = a' x_1 x_2, \\ & \quad a'= -2 a_{12} + 2\A[e_1] \left(\A^{-1}[e_2] (a_{11})\right) + 2\A[e_2] \left(\A^{-1}[e_1] (a_{22})\right).
	\end{split}
	\end{equation}
\textit{Definition of $P_4$}: We define $P_4$ dependent on $a'$ in \eqref{tuchel}. Note that $\B (a' x_1 x_2) =0$ and therefore $a' \in \ker (\B[e_1]) \cap \ker(\B[e_2])$. Then define
	\begin{align}\label{eq:nagelsmann}
	b_2 :=\tfrac{1}{2} \A^{-1}[e_1] a', \quad b_1 :=  \A^{-1}[e_1] (-\A[e_2] b_2).
	\end{align}
Note that $b_2$ is well-defined as $a' \in \ker( \B[e_1])$. Further, note that \begin{align} \label{favre}
	\B[e_1] (-\A[e_2] b_2) \stackrel{\eqref{eq:kloppo}}{=} \B[e_2] (\A[e_1] b_2) = \tfrac{1}{2}\B[e_2] a' =0.
	\end{align}
	Consequently $\A[e_2] b_2 \in \ker(\B[e_1])$ and so $b_1$ is well-defined. Let us set $P_4(x_1,x_2) :=(\tfrac{1}{3} b_1 x_1^3 + b_2 x_1^2 x_2)$. Then \begin{align*}
\A P_{4}(x_{1},x_{2}) = \A (\tfrac{1}{3}b_1 x_1^3 + b_2 x_1^2 x_2) &\hspace{0.2cm}= \left( \A[e_1] b_1  + \A[e_2] b_2\right) x_1^2 + 2 \A[e_1] b_2 x_1 x_2  \\
	  &\!\overset{\eqref{eq:nagelsmann}}{=}\left( -\A[e_2] b_2 + \A[e_2] b_2 \right) x_1^2 + a' x_1 x_2 \\ 
	  &\!\!\stackrel{\eqref{tuchel}_{1}}{=} 	(- S_1 u - \A P_3). 
	\end{align*}
We conclude that $\A P_3 + \A P_4 = - S_1 u$, which is what we wanted to show.

To summarise, we found $w \in \lebe^2(\torus_2;\R^d)$, such that $\A w = (u + S_1 u)+ S_2 u$ in $\mathrm{H}^{-1}(\torus_2;\R^d)$ and $P$ such that $\A P = - S_1 u$. Both $w$ and $P$ depend linearly on $u$. Let us now define \[
	\A^{-1} u := w +P.
	\]
	Then $\A (\A^{-1} u) = u + S_2 u$ in $\mathrm{H}^{-1}(\torus_2;\R^d)$. As $S_2 u$ is supported on $\partial Q$, we conclude that $\A (\A^{-1} u) = u$ in $\mathrm{H}^{-1}(Q;\R^d)$. 
	\end{proof}
	Using the result for first order operators, we are also able to formulate a version of Theorem \ref{thm:firstorder} for higher order operators.
	\begin{corollary}\label{thm:higherorder}
		Let $n=2$ and let $\B$ be a differential operator of order $k$. Then there exists a finite dimensional space $X \subset \mathrm{H}^k(Q;\R^d) \cap \ker(\B)$ consisting of polynomials and a linear, bounded map $\A^{-1}\colon \mathrm{H}^k(Q;\R^d) \cap \ker(\B) \to \lebe^2(Q;\R^m)$ such that $u-\A \circ \A^{-1} u \in X$. 
	\end{corollary}
	Essentially, the argument is that we can reduce this case to the case of first order operators. First of all, let us reduce to a first-order $\B$. Let $\B$ be of order $l \in \N$. Then $\mathbb{B}u =0$ if and only if $u^{l-1} = \nabla^{l-1} u$ satisfies 
	\begin{equation} \label{diffop:1}
	\B^{l-1} u^{l-1} =0 \quad \text{ and } \quad\curl^{l-1} u^{l-1} =0,
	\end{equation}
	where $\B^{l-1}$ is a suitable reformulation of the differential constraint $\B$ as a first order operator dependent on the $(l-1)$-derivatives; the condition $\curl^{l-1} u^{l-1}$ encodes that $u^{l-1}$ is a $(l-1)$-gradient. Observe that $\A_{l-1} := \nabla^{l-1} \circ \A$ is a potential for the differential operator described in~\eqref{diffop:1}. 	For $\A$ of order $k$ observe that $\A v = u$ if and only if for $v^{k-1} = \nabla^{k-1} v$ \begin{equation} \label{diffop:2}
	\A^{k-1} v^{k-1} = u \text{ and } \quad\curl^{k-1} v^{k-1} =0,
	\end{equation}
	where again, $\A^{k-1}$ is a suitable reformulation of $\A$ in terms of derivatives of order $(k-1)$. Taking \eqref{diffop:1} and \eqref{diffop:2} together and applying Theorem \ref{thm:firstorder}, up to a finite dimensional vector space, for each $u^{l-1}$ satisfying $\B^{l-1} u^{l-1} =0$ we might find $\tilde{v}$, such that \[
	(\A_{l-1})^{k+l-2} \tilde{v}=u, \quad \curl^{k+l-2} \tilde{v} =0.
	\]
	and, therefore, $v$, such that \[
	\nabla^{l-1} \circ \A v = u.
	\]
	As a consequence, up to a finite dimensional vector space $\mathcal{X}$, $\A v - u \in \mathcal{X}$.
\begin{remark}
To conclude, let us remark that another approach to the problem described in this section is discussed in \cite[Lem.~14]{Adolfo} for operators of maximal rank. Whereas we believe that our approach might also apply to other, slightly more general scenarios and since our focus here is more on displaying consequences of the constant rank conditions in the exemplary case of $n=2$, we shall defer the discussion to higher dimensions to future work. 
\end{remark}


\begin{thebibliography}{99}
	\bibitem{Adams} Adams, R.A.; Fournier, J.A.: Sobolev spaces. Second Edition. Pure and Applied Mathematics Series, Elsevier 2003. 
	\bibitem{Adolfo} Arroyo-Rabasa, A.: New projection and Korn estimates for a class of constant-rank operators on domains. ArXiv preprint, 
		arXiv:2109.14602.
		\bibitem{ArRaSi} Arroyo-Rabasa, A.; Simental, J.: An elementary proof of the homological properties of constant-rank operators. ArXiv preprint, arXiv:2107.05098v1. 
		\bibitem{BDG} Breit, D.; Diening, L.;  Gmeineder, F.: On the trace operator for functions of bounded A-variation. Anal.
		PDE, 13(2):559--594, 2020.
		\bibitem{Ciarlet} Ciarlet, P.G.; Ciarlet, P. Jr.: Another approach to linearized elasticity and a new proof of
Korn's inequality, Math. Models Methods Appl. Sci. 15(2005) 259-271.
		\bibitem{CMO20} Conti, S.; M\"{u}ller, S.;  Ortiz, M.: Symmetric div-quasiconvexity and the relaxation of static problems. Arch. Ration. Mech. Anal. 235, Issue 2:841--880, 2020.
		\bibitem{Cox} Cox, D., Little, J. and O'Shea, D.: Ideals, varieties, and algorithms. An introduction
to computational algebraic geometry and commutative algebra. Fourth edition. Undergraduate
Texts in Mathematics. Springer, Cham, 2015.
		\bibitem{DG} Diening, L.; Gmeineder, F.: Sharp trace and Korn inequalities for differential operators. ArXiv preprint, ArXiv 2105.09570.
		\bibitem{FoMu} Fonseca, I.; M\"{u}ller, S.: $\mathcal{A}$-quasiconvexity, lower semicontinuity and Young measures. SIAM J. Math. Anal., \textbf{30}(6):1355--1390, 1999. 
		\bibitem{FuchsSeregin} Fuchs, M.; Seregin, G.: Variational methods for problems from plasticity theory and for generalized
Newtonian fluids. Lecture Notes in Mathematics, 1749. Springer-Verlag, Berlin, 2000. vi+269 pp.
		\bibitem{GR19} Gmeineder, F.; Raita, B.:  Embeddings for A-weakly differentiable functions on domains. J. Func. Anal., 277(12):108278, 2019.
		\bibitem{GRV21} Gmeineder, F.; Raita, B.; Van Schaftingen, J.: Limiting trace inequalities for vectorial differential operators. Indiana Univ. Math. J., 2021.
		\bibitem{HNR} H\"{a}rk\"{o}nen, M.; Niklasson, L.; Raita, B.: Syzygies, constant rank, and beyond. ArXiv preprint:  arXiv:2112.12663.
		\bibitem{Hoermander} H\"{o}rmander, L.: Differentiability properties of solutions of systems of differential equations, Ark. Mat. 3 (1958), 527--535.
		\bibitem{Kala1} Kalamajska, A.: Coercive inequalities on weighted Sobolev spaces. Colloq. Math., LXVI(2):309--318, 1993.
		\bibitem{Mazya} Maz'ya, V.: Sobolev Spaces. Grundlehren der mathematischen Wissenschaften, Vol. 342. Second Edition, Springer, 2011.
		\bibitem{Murat} Murat, F.: Compacit\'{e} par compensation: condition necessaire et suffisante de continuit\'{e} faible sous une hypoth\'{e}se de rang constant, Ann. Scuola Norm. Sup. Pisa Cl. Sci. (4), \textbf{8} (1981),
		pp. 68--102.
		\bibitem{Raita} Raita, B.: Potentials for $\mathscr{A}$-quasiconvexity. Calc. Var. (2019) 58:105.
		\bibitem{raitanew} Raita, B.:  A simple construction of potential operators for compensated compactness, ArXiv preprint: arXiv:2112.11773.
		\bibitem{WiSc} Schulenberger, J.R.; Wilcox, C.H.: Coerciveness inequalities for nonelliptic systems of partial differential equations. Annali di Matematica, 88 (1971), pp. 229--305.
		\bibitem{Smith} Smith, K.T.: Formulas to represent functions by their derivatives. Math. Ann., 188:53--77, 1970.
		\bibitem{Spencer} Spencer, D.C.:  Overdetermined systems of linear partial differential equations, Bull. Amer. Math. Soc. 75 (1969), 179--239.
	\end{thebibliography}
\end{document}